\documentclass[a4paper,11pt]{article}

\usepackage{soul}

\usepackage{makeidx}
\usepackage[T1]{fontenc}
\usepackage[dvips]{graphicx}
\usepackage{amsmath,amscd,amsthm}
\usepackage{amssymb}
\usepackage{tabularx}
\usepackage[dvips]{graphics}
\usepackage[latin1]{inputenc}

\makeindex

\usepackage[dvipsnames,usenames]{color}

\newtheorem{lmm}{Lemma}
\newtheorem{thr}{Theorem}
\newtheorem{dfn}{Definition}
\newtheorem{rmr}{Remark}

\newtheorem{notation}{Notation}

\newcommand{\R}{\mathbb{R}}
\newcommand{\C}{\mathbb{C}}

\newcommand{\fkg}{\mathfrak g}
\newcommand{\fkh}{\mathfrak h}

\usepackage{color}

\begin{document}

\begin{center} {\Large \bf Filiform Lie algebras with low derived length}
\end{center}

\begin{center}
{\bf F.J. Castro-Jim\'{e}nez, M. Ceballos, J. N\'u\~nez-Vald\'es}\footnote{ {\small FJCJ: Departamento de \'Algebra e IMUS. Facultad de Matem\'aticas, Universidad de Sevilla. C/ Tarfia s/n, 41012 Seville (Spain). {\bf castro@us.es} \\ MC: Dpto. de Ingenier\'{\i}a. Universidad Loyola Andaluc\'{\i}a, Campus Palmas Altas, C/ Energ\'{\i}a Solar 1, Ed. E, Seville (Spain).  {\bf mceballos@us.es} \\ JNV: Departamento de Geometr\'{\i}a y Topolog\'{\i}a. Facultad de Matem\'aticas, Universidad de Sevilla. C/ Tarfia s/n, 41012 Seville (Spain). {\bf jnvaldes@us.es}}} \end{center}

\begin{abstract}
We construct, for any $n\geq 5$, a family of complex filiform Lie algebras with derived length at most $3$ and dimension $n$.
We also give examples of $n$-dimensional filiform Lie algebras with derived length greater than $3$.
\end{abstract}

\vspace{0.3cm}



\section{Introduction}
\label{intro}

The derived length, also known as solvability index, of nilpotent groups or nilpotent Lie algebras has been studied by a number of authors. In the case of finite $p$--groups this study was initiated by Burnside \cite{Burnside1913, Burnside1914} and then continued by several authors (e.g. P. Hall, W. Magnus, N. Itô, G. Higman, N. Blackburn, A. Mann among many others). In the case of a nilpotent Lie algebra, the study of its derived length has been treated in several papers, e.g. by Dixmier \cite{Dixmier1955}, Patterson \cite{Patterson1955, Patterson1956}, Bokut \cite{Bokut1971}. In \cite{B-D-V-2008} the authors show that there are filiform Lie algebras of arbitrary derived length. They describe, for any $k\geq 2$, a filiform Lie algebra of derived length $k$  and dimension $n$ for each $n$ satisfying $2^{k} \leq n+1 < 2^{k+1}$. These algebras were studied by Y. Benoist in \cite{Benoist1995}.

Filiform Lie algebras were introduced by  M. Vergne \cite{VE} and they have been classified up to dimension $8$; see \cite{AG} and a corrected version of this classification given in \cite{Remm}. Moreover, nilpotent Lie algebras are classified up to dimension $7$ in \cite{Gong}. In view of the classification of filiform Lie algebras, several invariants have been introduced in the literature. For example, in \cite{ENR1} two numerical invariants were introduced and studied. Our results on the derived length of filiform Lie algebras are based on these invariants. Recently, some new invariants have been defined in order to study these algebras. In particular, the notion of {\em breath} is introduced in \cite{KMS} and a generalization of it, by the so-called {\em characteristic sequence}, in \cite{Remm2}. In addition, Lie algebras graded by an abelian group are studied in \cite{BZ}. The special case of graded filiform Lie algebras is treated in \cite{BGR,BGR2,BGR3}.

In this paper we construct, for any $n\geq 5$, a family of complex filiform Lie algebras with derived length at most $3$ and dimension $n$, see Theorem \ref{solvability-index-leq4}. 
To this end we improve the result \cite[Prop. 2]{CNT2} and describe the law of any filiform Lie algebra of dimension $n\geq 5$ with respect  to a suitable {\em adapted basis} by using two numerical invariants associated with the algebra, see Theorem \ref{leygeneral}. This result is related to the description of the affine variety of $n$--dimensional filiform Lie algebras given in \cite[Sec. 4]{Million}, where the variety of Lie algebras of {\em maximal class} is described. See also \cite{SZ} where the related notion of {\em narrow} Lie algebras is treated.

Theorem \ref{leygeneral} generalizes a result of F. Bratzlavsky \cite{Bra} valid for metabelian Lie algebras.  We also construct a family of complex filiform Lie algebras of dimension $15$ and derived length $4$ generalising the one given in \cite[Ex. 3.2]{BUarxiv}. Let us remark that, as proved in \cite[Prop. 3.7]{BUarxiv}, there is no complex filiform Lie algebra of dimension less than or equal to 14 and derived length 4.

Finally, let us also point out that there are several results concerning solvable Lie algebras with low derived length. A solvable Lie algebra is said to be $k$-step solvable, see Definition \ref{defresoluble}, if its derived length is $k$. In particular, $1$-step solvable Lie algebras are the abelian ones. Abelian subalgebras and ideals are useful in the study of Lie algebra contractions and degenerations. There is some extensive literature on these topics, in particular for low-dimensional Lie algebras; see \cite{BU10,GOR,GRH} and the references given therein. Next, $2$-step solvable Lie algebras are called metabelian. These algebras have been studied in \cite{AC} by using the notion of weight graphs. Moreover, they admit abelian complex and Novikov structures, see \cite{BaD,BD}.
In \cite{U} the author uses $3$-step solvable Lie algebras to construct a homogeneous conformally parallel Spin(7) metric on a certain solvmanifold. Finally, it is proved in \cite{NZ} that $3$-step solvable Lie algebras contain the first oscillator algebras with non-trivial semi-equicontinuous coadjoint orbits.

\section{Preliminaries}
\label{sec:1}

In this section we recall some preliminary concepts, results and notations on Lie algebras. We have mainly followed \cite{GK,JAC,Serre,Var98}. From here on, only finite-dimensional complex Lie algebras are considered.  

Given a Lie algebra $\mathfrak{g}$, a vector subspace
$\mathfrak{h}$ of $\mathfrak{g}$ is an {\em ideal}
if $ [\mathfrak h,\mathfrak g] \subseteq \mathfrak{h}$. We denote by $\dim \fkh$ the dimension of $\fkh$ as a vector space.

The {\it descendant central series} of ideals, also known as the {\it lower central series}, of a given Lie algebra $\mathfrak{g}$ is the filtration $$C^1\fkg
\supseteq C^2\fkg \supseteq
\ \dots \ \supseteq
C^k\fkg \supseteq \dots$$
where
$$C^1\fkg=
\mathfrak{g} \, {\mbox{ and }} \,
C^k\fkg=[C^{k-1}\fkg,
\mathfrak{g}] \, {\mbox{ for all }} k\geq 2. $$ One has $[C^k\fkg, C^\ell\fkg] \subseteq C^{k+\ell}\fkg$ for all $k, \ell \geq 1$.

\begin{dfn}\label{defnilpotente}
A Lie algebra $\mathfrak{g}$ is said to be {\em nilpotent} if there exists $m\in \mathbb N$
such that $C^m\fkg = \{0\}$. The smallest such $m$ is called the {\em nilpotency class} (or the {\em nilindex}) of $\fkg$.
\end{dfn}

The {\em derived series} of a given Lie algebra $\mathfrak{g}$ is the filtration $$D^0\fkg
\supseteq D^1\fkg \supseteq \ \dots \ \supseteq
D^k\fkg \supseteq \dots$$
where
$$D^0\fkg=
\mathfrak{g} \, {\mbox{ and }} \,
D^k\fkg=[D^{k-1}\fkg,
D^{k-1}\fkg] \, {\mbox{ for all }} k\geq 1. $$

\begin{dfn}\label{defresoluble}
A Lie algebra $\mathfrak{g}$ is said to be {\em solvable} if there exists $m\in \mathbb N$  such that $D^m\fkg = \{0\}$. The smallest such $m$ is called the {\em derived length} (or the {\rm solvability index}, or even the {\em solvindex}) of $\fkg$. We say that $\fkg$ is {\em $m$-step solvable} if the {\em derived length} of $\fkg$ is  $m$.
\end{dfn}

The {\em derived} Lie algebra of $\mathfrak{g}$ is by definition $C^2 \fkg=D^1 \fkg=[\fkg,\fkg]$. From here on, we will denote the derived algebra by $D \fkg$. A Lie algebra $\mathfrak{g}$ satisfying $D \fkg=\{0\}$ is called {\em abelian}. If $D \fkg$ is abelian, i.e. if $D^2 \fkg=\{0\}$, then $\fkg $ is said to be {\em metabelian}.

Let us notice that every nilpotent Lie algebra is solvable, since $D^{k-1}\fkg \subseteq
C^k\fkg$ for all $k \geq 1$. Moreover, we have
$D^{k} \fkg \subset C^{2^k} \fkg$, see e.g. \cite[page 25]{JAC}. This means in particular that if $\dim \fkg \leq 2^k$, for some integer $k$,  then the derived length of $\fkg$ is less than or equal to $k$.

\begin{dfn}\rm{\cite[Section 1.5]{VE}}\label{deffiliforme}
A $n$-dimensional Lie algebra $\mathfrak{g}$ is said to be {\em filiform} if its descendant central series satisfies
\begin{equation}\label{lcs2}
{\rm dim}(C^k \, \fkg)\!=\!n-k {\mbox{ for all }} 2 \leq k\leq n.
\end{equation}
\end{dfn}

Any $n$--dimensional filiform Lie algebra is nilpotent and its nilpotency class is $n$.

A basis $\{x_1, \ldots,x_n\}$ of  a  filiform Lie algebra $\mathfrak{g}$ is called {\em adapted}, see \cite[Sec. 4.2]{VE}, if the following relations hold
\begin{equation} \label{eqdeVergne}\begin{array}{ll}
\ [x_1, x_i] = x_{i+1} & \,\,\, {\rm for} \,\,\, 2 \leq i \leq n-1, \\
\ [x_1, x_n] = 0, \\
\ [x_2,x_3] = 0 & \,\,\,{\rm mod} \,\,\, {\C}x_5 + \cdots + {\C}x_n.
\end{array}
\end{equation}
As a consequence, it is also satisfied that
\begin{equation} \label{Vergne2} [x_i,x_j]=0 \quad {\rm mod} \quad \sum_{k \geq i+j}^n {\C} x_k.
\end{equation}

If we apply the basis change
\begin{equation} \label{basischange}
e_1=x_1, \quad e_2=x_n, \quad e_3=x_{n-1}, \quad \ldots \quad e_n=x_2
\end{equation}
then the new basis satisfies the following relations
\begin{equation} \label{eq2}\begin{array}{ll}
 \ [e_1, e_h] = e_{h-1} & \,\,\, {\rm for} \,\,\, 3 \leq h \leq n, \\
\ [e_2, e_h] = 0 & \,\,\, {\rm for} \,\,\, 1 \leq h \leq n, \\
\ [e_3,e_h] = 0 & \,\,\,{\rm for} \,\,\, 2 \leq h \leq n.
\end{array}
\end{equation}

In order to simplify our presentation we say that a basis $\{e_1, \ldots, e_n\}$ of a filiform Lie algebra is adapted if it satisfies (\ref{eq2}). 

As it was pointed out in \cite[Section 4.2]{VE} any filiform Lie algebra admits an adapted basis. In fact, with respect to an adapted basis, and as a consequence of (\ref{eqdeVergne}) and (\ref{Vergne2}), the  ideals of the descendant central series are given by
\begin{equation}\label{coroderivada} C^k \, \fkg = \langle e_2, \ldots, e_{n-k+1}  \rangle   \end{equation} for $2\leq k\leq n-1$, where the angle brackets mean the $\C$--vector space generated by the corresponding vectors.

A $n$-dimensional filiform Lie algebra $\mathfrak{g}$ is called a {\em model} filiform Lie algebra (see \cite{GK}) if
the only nonzero brackets  in its law  are $[e_1, e_h] = e_{h-1}$,
for $3 \leq h \leq n$, where $\{e_1,\ldots,e_n\}$ is an adapted basis of $\mathfrak{g}$.
This condition is obviously independent of the chosen adapted basis in $\mathfrak{g}$. Notice that for $n \leq 4$ every filiform Lie algebra is a model filiform Lie algebra.

\subsection{Two Numerical Invariants of Filiform Lie Algebras}\label{TwoNumerical}

We recall the definition of two invariants of a filiform Lie algebra $\fkg$ introduced in~\cite{ENR1}. These invariants are defined for non-model filiform Lie algebras, so we can assume that $n > 4$. 

First, the
invariant $z_1=z_1(\fkg)$ is defined as  $$z_1 = {\rm max} \{k \in {\mathbb N} \,| \,C_{\mathfrak{g}}
(C^{n-k+2}\fkg) \supseteq C^2\fkg \}$$
where $C_{\mathfrak{g}}
(\mathfrak{h})$ is the centralizer of a given Lie subalgebra $\mathfrak{h}$ of $\mathfrak{g}$, i.e. the set of elements in $\fkg$ whose bracket with any element of $\mathfrak{h}$ is zero.

The invariant $z_2=z_2(\fkg)$  is defined as  $z_2 = \, {\rm max} \, \{ k
\in {\mathbb N} \,|\, C^{n-k+1}\fkg\,\, {\rm is} \,\,
{\rm abelian}\}$. Consequently, $C^{n-z_2+1}\fkg$
is the largest abelian ideal in the lower central series of $\mathfrak g$. This invariant is related to the notion of $k$-abelian filiform Lie algebra given in \cite{GGK}. More concretely, every filiform Lie algebra is
$(n-z_2+1)$-abelian.

\begin{rmr}\label{cason-1}
Notice that, for any $n$-dimensional non-model filiform Lie algebra $\fkg$, one has $z_2(\fkg)\leq n-1$. Moreover, $z_2(\fkg)=n-1$ if and only if $\fkg$ is metabelian.
\end{rmr}

Equivalent definitions for the invariants $z_1$ and $z_2$, more  appropriate for practical use, are:\, $z_1 = \, {\rm min}
\, \{k \geq 4 \, | \, [e_k,e_n] \neq 0\}$ and $z_2 = \, {\rm min}
\, \{k \geq 4 \, | \, [e_k,e_{k+1}] \neq 0\}$, where $\{e_1, \ldots, e_n\}$
is an adapted basis of $\mathfrak g$.

These two invariants satisfy the following inequalities, see \cite[Th. 15]{ENR1}:
\begin{equation}\label{Bound}
4 \leq z_1 \leq z_2 < n \leq 2z_2-2.
\end{equation}

Given a $n$-dimensional non-model filiform Lie algebra $\mathfrak g$, we use the triple $(z_1,z_2,n)$ to summarize the information about both invariants and the dimension of $\mathfrak g$.

\subsection{Previous results}\label{previousresults}

In \cite[Lemme 1]{Bra} Bratzlavsky obtained the general law for a filiform metabelian Lie algebra, that is, a filiform Lie algebra associated with the triple $(z_1,n-1,n)$, for $4\leq z_1 \leq n-1$. In \cite[Sec. 1]{L} the author improves, for the infinite dimensional case, the classification result of \cite[Prop. 3]{Bra} only valid in finite dimension. On the other hand, in \cite[Lemma 1.2]{GJK} the authors give the parametric expression of a list of basic brackets for any $n$-dimensional filiform Lie algebra.

\begin{thr}{\rm{ \cite[Lemme 1]{Bra} }}\label{Teorema-de-Bratz}
Let $\fkg$ be a filiform Lie algebra of dimension $n \geq 5$ whose derived Lie algebra $D \fkg$ is abelian. Then there exist a basis $\{x_1, \ldots, x_n\}$ of $\fkg$ and some complex numbers $\lambda_0, \ldots, \lambda_{n-5}$ such that
\begin{align*} [x_1,x_i] &=x_{i+1} \,\,\, {\rm for} \,\, 2 \leq i \leq n-1; \\ 
[x_i,x_j] &= 0 \,\,\, {\rm for} \,\, 3 \leq i < j \leq n \quad {\rm and}\\
[x_2,x_i] & = \sum_{r=0}^{n-i-2} \lambda_r x_{i+2+r} \,\,\, {\rm for} \,\, 3 \leq i \leq n-2.\end{align*}
\end{thr}

\begin{rmr}
Notice that when $\lambda_0=\cdots= \lambda_{n-5}=0$, we obtain the model filiform Lie algebra since the equalities given in (\ref{basischange})
define an adapted basis of $\fkg$.

The parameters $\lambda_r$ in Theorem \ref{Teorema-de-Bratz} are free. This is due to the fact that $D \fkg$ is abelian and, therefore, all the Jacobi identities are satisfied for any value of the $\lambda_r$.
\end{rmr}

\begin{notation} \label{Ph} Given a $n$-dimensional Lie algebra $\fkg$ with basis
$\{e_1, \ldots,e_n\}$ and an integer $1\leq h \leq n$  we denote by
$$P_h : \mathfrak{g} \rightarrow {\mathbb C}$$ the ${\C}$-linear map defined as follows: for any vector $u\in \mathfrak{g}$, $P_h(u)$ is the $h$-th coordinate of $u$ with respect to the given  basis; i.e. one has $$u = \sum_{h=1}^{n} P_h(u) e_h.$$\end{notation}

The following theorem is stated in \cite[Prop. 2]{CNT} but we correct here a mistake there on the coefficient of $e_2$ in the brackets $[e_{z_{1}+k},e_{z_{2}+\ell}]$. More precisely, next theorem describes the law of any filiform non-model Lie algebra, with a given associated triple $(z_1,z_2,n)$ with respect to a suitable adapted basis. This result can be compared to the description of the affine variety of $n$--dimensional filiform Lie algebras given in \cite[Sec. 4]{Million}.

\begin{thr}{\rm {\cite[Prop. 2]{CNT}}} \label{leygeneral}  Let $\mathfrak{g}$ be a $n$-dimensional non-model filiform Lie
algebra with associated triple $(z_1,z_2,n)$. Then, there exist an adapted basis $\{e_1, \ldots, e_n\}$ of $\fkg$
and some complex numbers $\alpha_i$, $\gamma_j$  and $\beta_{k \ell}$, with $1 \leq i \leq z_2-z_1+1$, $1 \leq j \leq n-z_2-1$,
$2 \leq \ell \leq n-z_{2}$, and $1 \leq k < z_{2}-z_{1}+\ell$, such that
\begin{align*} [e_1,e_h]&= e_{h-1} \mbox{ \rm{ for }} 3 \leq h \leq n, \\
[e_{z_1+i},e_{z_2+1}]&=\alpha_1 e_{i+2}+\alpha_2 e_{i+1}+\cdots+\alpha_{i+1} e_2 \,\, \mbox{ \rm{ for }} 0 \leq i\leq z_2-z_1, \\
[e_{z_1},e_{z_2+j}]&=\alpha_1 e_{j+1}+\gamma_1\, e_j+ \cdots+\gamma_{j-1}\,e_2 \,\,  \mbox{ \rm{ for }} 2 \leq j \leq n-z_2,  \\
[e_{z_{1}+k},e_{z_{2}+\ell}]&= \sum_{h=2}^{k+\ell} P_h\left([e_{z_{1}+k-1},e_{z_{2}+\ell}] +[e_{z_{1}+k},e_{z_{2}+\ell-1}]\right) e_{h+1}+ \beta_{k \ell} \, e_2, \\ &\phantom{= }{\mbox{\rm{\ for }}} 2 \leq \ell \leq n-z_{2},\,\,  1 \leq k < z_{2}-z_{1}+\ell.\\
\end{align*}
\end{thr}

\begin{proof}

The proof of this theorem follows the one of \cite[Prop. 2]{CNT} bearing in mind that at the end of that proof, the coefficient $\beta_{k \ell}$ of $e_2$ in the bracket $[e_{z_1+k},e_{z_2+\ell}]$ is in general different from the coefficients $\{\alpha_i\}$ and $\{\gamma_j\}$.

\end{proof}

\begin{rmr}
Notice that the Jacobi identity induces quadratic relations among the parameters $\alpha_i$, $\beta_{k\ell}$ and $\gamma_j$. These quadratic relations define a certain non-empty Zariski closed set $\mathcal Z$ in the affine space $\C^\mu$ of the parameters. Here the number of parameters is given by
$$\mu = (z_2-z_1+1) + (n-z_2-1) + \frac{(n-z_2-1)(n+z_2-2z_1)}{2}=\frac{(n-z_2)(n+z_2-2z_1+1)}{2}.$$ Moreover, the Lie algebras in Theorem \ref{leygeneral} must be non-model and with fixed numerical invariants $(z_1,z_2,n)$. These last conditions define a certain non-empty Zariski open set $\mathcal{F}$ of the previous closed set $\mathcal Z$. See Remark \ref{dim} for a comment on the dimension of this quasi-affine variety.
\end{rmr}

\begin{notation}\label{Fabc}

Notice that, by definition,  the Lie algebras in the family $\mathcal{F}$ have a fixed associated triple $(z_1,z_2,n)$.

According to Theorem \ref{leygeneral}, the set of laws of $n$--dimensional non-model filiform Lie algebras is the disjoint union of the sets $\mathcal{F}$ when $(z_1,z_2,n)$ varies in the index set defined by inequalities (\ref{Bound}); that is, when $4 \leq z_1 \leq z_2 < n \leq 2z_2-2.$

\end{notation}

\begin{rmr}
Let us notice that Theorem \ref{leygeneral} implies Theorem \ref{Teorema-de-Bratz} for non-model filiform Lie algebras.
\end{rmr}

\begin{rmr}\label{recursivePh}
According to Notation \ref{Ph} and Theorem \ref{leygeneral}, the linear maps $P_h$ satisfy the following relation
$$P_h([e_i,e_j])=P_{h-1}([e_{i-1},e_j]+[e_{i},e_{j-1}]) \quad {\rm for} \, \,\, 1 < i < j \leq n.$$
Previous equality will be used in the next section.
\end{rmr}

\section{Main results}
\label{Main}

This section contains the main new results of this work. We state four Lemmas that are used to prove that the derived length of a family of finite dimensional filiform Lie algebra is at most $3$ (see Theorem \ref{solvability-index-leq4}). This family will be described in Notation \ref{Fac}.

We start with Lemma \ref{cason-2} in order to study the case where $z_2 = n-2$.

\begin{lmm}\label{cason-2}

Let $\mathfrak{g}$ be a $n$-dimensional non-model filiform Lie algebra with $z_2=n-2$. Then $n \geq 6$ and $\fkg$ has derived length $3$.

\end{lmm}

\begin{proof}

First, the condition $n\geq 6$ follows from the fact that $\fkg$ is a non-model filiform Lie algebra and from inequalities (\ref{Bound}). Moreover, $\mathcal{D}^2 \fkg$ is nonzero since $z_2 =n-2$ (see Remark \ref{cason-1}).
According to Theorem \ref{leygeneral}, we have the following general law for the algebra $\mathfrak{g}$, with respect to a suitable adapted basis
\begin{align*} [e_1,e_h]&= e_{h-1}, \,\, 3 \leq h \leq n, \\
[e_{z_1+i},e_{n-1}]&=\alpha_1 e_{i+2}+\alpha_2 e_{i+1}+\cdots + \alpha_{i+1}e_2, \,\, 0 \leq i \leq n-z_1-2,\\
[e_{z_1},e_{n}]&=\alpha_1 e_{3}+\gamma_1 e_2,\\
[e_{z_1+k},e_{n}]&= \sum_{h=2}^{k+2} P_h([e_{z_1+k-1},e_{n}] +[e_{z_1+k},e_{n-1}])
e_{h+1}+ \beta_{k 2} \, e_2, \,\, 0< k < n-z_1.
\end{align*}

According to equality (\ref{coroderivada}), $D^2\fkg \subseteq \langle e_2, \ldots, e_{n-z_1} \rangle$ is abelian since $n-z_1 < n-1$ and then $D^3\fkg=\{0\}$.
\end{proof}

\begin{notation}\label{pqrRkSk}
We will use the following notations
$$\begin{array}{lll}
p=2n-z_1-2z_2-3 & q=n-z_2-2 & r=n-z_1-1\\ & & \\
\displaystyle{R_k=\biggr(\binom{r}{k}-\binom{r}{k-1}\biggr)\binom{q}{k}}& \displaystyle{S_k=\biggr(\binom{r}{k}-\binom{r}{k-1}\biggr)\binom{q}{k-1}} & \\
\end{array}$$
Given a real number $y$ the expression
$\lfloor y \rfloor$ denotes the largest integer less than or equal to $y$.
\end{notation}

For the next three Lemmas, we will refer to the coefficients $\alpha_i$, $\beta_{k \ell}$, $\gamma_j$ introduced  in Theorem \ref{leygeneral}. If no confusion arises we  use $\beta_{k,\ell}$ instead of $\beta_{k \ell}$. Moreover, in the proofs of Lemmas \ref{alpha1es0} and \ref{caso(z1+m,n-1,n)}, we will use the following property that is proved in \cite[Lemma 4]{HOUS}:
$$[e_i,e_j]=0 \quad {\rm for} \,\, 1 < i < z_1, \,\, 1 < j.$$

\begin{lmm}{\rm \cite[Prop. 6]{CNT}} \label{alpha1es0}
Let $\mathfrak{g}$ be a $n$-dimensional non-model filiform Lie algebra. If the derived algebra $\mathcal{D} \fkg$ is not abelian, then $[e_{z_1},e_{z_2+1}]=\alpha_1 e_2 = 0$.
\end{lmm}

\begin{proof}
We give here an alternative proof to the one in  \cite[Prop. 6]{CNT}, which is related to other results of this section. We assume that $\fkg$ is a non-model filiform Lie algebra with associated triple $(z_1,z_2,n)$. Let us consider the Jacobi identity $$J(e_{z_1},e_{n-1},e_n) =  [[e_{z_1},e_{n-1}],e_n]+[[e_{n-1},e_{n}],e_{z_1}]+[[e_{n},e_{z_1}],e_{n-1}] = 0.$$
Notice first that by Theorem \ref{leygeneral}, $$[e_{z_1},e_{n-1}]=\alpha_1 e_{n-z_2}+\gamma_1 e_{n-z_2-1}+ \cdots + \gamma_{n-z_2-2}e_2, \qquad [e_{z_1},e_{z_2+1}]=\alpha_1 \, e_2.$$
Therefore the first term of the Jacobi identity is given by $$[[e_{z_1},e_{n-1}],e_n]=\alpha_1 [e_{n-z_2},e_n]+\gamma_1 [e_{n-z_2-1},e_n]+ \cdots + \gamma_{n-z_2-2}[e_2,e_n]=$$
$$\alpha_1 [e_{n-z_2},e_n]+\gamma_1 [e_{n-z_2-1},e_n]+ \cdots + \gamma_{n-z_2-z_1}[e_{z_1},e_n].$$

Next, for the second term, we have
$$[e_{n-1},e_n]=\sum_{h=2}^{q+r+2} P_h([e_{n-2},e_n])e_{h+1} + \beta_{n-1-z_1,n-z_2} \, e_2,$$ so
$$[e_{z_1},[e_{n-1},e_n]]=\sum_{h=2}^{q+r+2} P_h([e_{n-2},e_n])[e_{z_1},e_{h+1}] + \beta_{n-1-z_1,n-z_2} \, [e_{z_1},e_2] $$
$$=\sum_{h=z_1}^{q+r+2} P_h([e_{n-2},e_n])[e_{z_1},e_{h+1}].$$
Finally, for the last term, we obtain
$$[e_{z_1},e_{n}]=\alpha_1 e_{n-z_2+1}+\gamma_1 e_{n-z_2}+\cdots+\gamma_{n-z_2-1}e_2$$
$$[[e_{z_1},e_n],e_{n-1}]=\alpha_1 [e_{n-z_2+1},e_{n-1}]+\gamma_1 [e_{n-z_2},e_{n-1}]+\cdots+\gamma_{n-z_2-1}[e_2,e_{n-1}]]=$$
$$\alpha_1 [e_{n-z_2+1},e_{n-1}]+\gamma_1 [e_{n-z_2},e_{n-1}]+\cdots+\gamma_{n-z_1-z_2+1}[e_{z_1},e_{n-1}].$$
In this way, $J(e_{z_1},e_{n-1},e_n)=0$ is equivalent to
$$\alpha_1 [e_{n-z_2},e_n]+\gamma_1 [e_{n-z_2-1},e_n]+ \cdots + \gamma_{n-z_2-z_1}[e_{z_1},e_n]-\sum_{h=z_1}^{q+r+2} P_h([e_{n-2},e_n])[e_{z_1},e_{h+1}]$$
$$-\alpha_1 [e_{n-z_2+1},e_{n-1}]-\gamma_1 [e_{n-z_2},e_{n-1}]-\cdots-\gamma_{n-z_1-z_2+1}[e_{z_1},e_{n-1}] = 0.$$
Since $\mathcal{D} \fkg$ is not abelian, $z_2 \leq n-2$ and, therefore, $z_1<n-1$. Consequently, $q \geq 0$ and $p,r \geq 1$, where $p, q$ and $r$ were introduced in Notation \ref{pqrRkSk}. Notice also that $p+4 \leq n$.
The coefficient of $e_{p+4}$ in the Jacobi identity $J(e_{z_1},e_{n-1},e_n)=0$ is given by
$$\alpha_1 \biggr\{ P_{p+3}([e_{q+1},e_n]+[e_{q+2},e_{n-1}])- P_{q+r+2}([e_{n-2},e_n])-$$
$$ P_{p+3}([e_{q+2},e_{n-1}]+[e_{q+1},e_{n-2}])\biggr\}=$$
$$\alpha_1 \biggr\{ P_{p+4}([e_{q+2},e_n])-P_{q+r+2}([e_{n-2},e_n])-P_{p+4}([e_{q+3},e_{n-1}])\biggr\}.$$

Now, from the recursive definition of the linear maps $P_h$ (see Remark \ref{recursivePh}) and Vandermonde's identity for combinatorial numbers, we obtain
$$P_{n-z_2+1}([e_{z_1+k},e_{n-k}])=\binom{q+1}{k} \alpha_1.$$
Consequently,
$$P_{p+4}([e_{q+2},e_n])=   \binom{p+2}{r-z_2+1} \alpha_1$$
$$P_{q+r+2}([e_{n-2},e_n])=   \binom{q+r}{r-1} \alpha_1$$
$$P_{p+4}([e_{q+3},e_{n-1}])=  \binom{p+2}{r-z_2+2} \alpha_1.$$
The previous equalities only hold if the corresponding Lie brackets are non-zero.
We conclude that the coefficient of $e_{p+4}$ in the Jacobi identity is given by
$$\alpha_1^2 \biggr\{ \binom{p+2}{r-z_2+1}-\binom{q+r}{r-1}- \binom{p+2}{r-z_2+2}\biggr\}=$$
$$\alpha_1^2 \biggr\{-\frac{z_1-2}{q+1} \binom{p+2}{q}-\binom{q+r}{r-1}\biggr\}.$$
The previous expression is zero if and only if $\alpha_1=0$.

\end{proof}

\begin{rmr}

The converse of Lemma \ref{alpha1es0} is also true. Let $\fkg$ be a $n$-dimensional non-model filiform Lie algebra associated with the triple $(z_1,z_2,n)$ and verifying $[e_{z_1},e_{z_2+1}] =0$. According to the definition of the invariants $z_1,z_2$ and Theorem \ref{leygeneral}, $z_2 \leq n-2$ and the following bracket must be non-zero
$$[e_{z_2},e_{z_2+1}]=\alpha_2 e_{z_2-z_1+1}+\cdots + \alpha_{z_2-z_1+1}e_2$$
Therefore, $\mathcal{D} \fkg$ is not abelian.

\end{rmr}

\begin{notation} \label{ambmcm} We recall here the expressions introduced in Notation \ref{pqrRkSk} and introduce some more useful notations
$$\begin{array}{ll} a_m = & {\binom{m+q-1}{m} \binom{m+p}{q}-\binom{m+q}{m} \binom{m+p}{q-1}
-\binom{p+m}{m}\displaystyle\sum_{k=0}^{\lfloor\frac{r}{2}\rfloor} R_k} ,\\ & \\
b_m=& {\binom{m+q-1}{m} \binom{m+p}{q+1}-\binom{m+q-1}{m-1} \binom{m+p}{q}-\binom{m+q}{m} \binom{m+p}{q}} \\ & \\ & {
-\binom{m+q}{m-1} \binom{m+p}{q-1}- \binom{m+p}{m-1} \sum_{k=0}^{\lfloor\frac{r}{2}\rfloor} R_k - \binom{m+p}{m} \displaystyle\sum_{k=0}^{\lfloor\frac{r}{2}\rfloor} S_k} ,\\ & \\
c_m= & \binom{m+q-1}{m-1}\binom{m+p}{q+1}-\binom{m+q}{m-1}\binom{m+p}{q}- \binom{m+p}{m-1}
\displaystyle\sum_{k=0}^{\lfloor\frac{r}{2}\rfloor} S_k , \\
s= &\min\{q+1,r-q-1\}. \\
\end{array} $$
\end{notation}

\begin{lmm}\label{caso(z1+m,n-1,n)}

Let $\fkg$ be a $n$-dimensional non-model filiform Lie algebra associated with the triple $(z_1,z_2,n)$, whose derived algebra $\mathcal{D} \,\fkg$ is not abelian. Then, for \, $m=0, \ldots, s$, the coefficient of the vector $e_{m+p+2}$, where $p=2n-z_1-2z_2-3$,  in the Jacobi identity $J(e_{z_1+m},e_{n-1},e_n)=0$  is
$$a_m \gamma_1^2+b_m \gamma_1 \alpha_2 + c_m \alpha_2^2.$$

\end{lmm}

\begin{proof}
First, from Lemma \ref{alpha1es0}, we can suppose that $\alpha_1=0$ in Theorem \ref{leygeneral}. Next, by using that theorem, Remark \ref{recursivePh} and Vandermonde's identity for combinatorial numbers, we obtain the following relation
\begin{equation}\label{comb2}
P_{n-z_2}([e_{z_1+k},e_{n-k}])=\binom{q}{k} \gamma_1+\binom{q}{k-1} \alpha_2
\end{equation}
for $0 \leq k \leq n-1$. Now, let us consider the Jacobi identity $$J(e_{z_1+m},e_{n-1},e_n) =  [[e_{z_1+m},e_{n-1}],e_n]+[[e_{n-1},e_{n}],e_{z_1+m}]+[[e_{n},e_{z_1+m}],e_{n-1}] = 0.$$
Notice first that from Theorem \ref{leygeneral} we have $$[e_{z_1+m},e_{n-1}]=\sum_{h=2}^{m+q}P_h([e_{z_1+m-1},e_{n-1}]+[e_{z_1+m},e_{n-2}])e_{h+1}+\beta_{m,n-z_2-1}e_2$$ and then  $$[[e_{z_1+m},e_{n-1}],e_n]=\sum_{h=2}^{m+q}P_h([e_{z_1+m-1},e_{n-1}]+[e_{z_1+m},e_{n-2}])[e_{h+1},e_n]+\beta_{m,n-z_2-1}[e_2,e_n]$$
$$=\sum_{h=z_1-1}^{m+q}P_h([e_{z_1+m-1},e_{n-1}]+[e_{z_1+m},e_{n-2}])[e_{h+1},e_n].$$

Next, for the second term, we have
$$[e_{n-1},e_n]=\sum_{h=2}^{q+r+1} P_h([e_{n-2},e_n])e_{h+1} + \beta_{n-z_1-1,n-z_2} \, e_2,$$ so
$$[e_{z_1+m},[e_{n-1},e_n]]=\sum_{h=2}^{q+r+1} P_h([e_{n-2},e_n])[e_{z_1+m},e_{h+1}] + \beta_{n-z_1-1,n-z_2} \, [e_{z_1+m},e_2] $$
$$=\sum_{h=z_1}^{q+r+1} P_h([e_{n-2},e_n])[e_{z_1+m},e_{h+1}].$$

Finally, for the last term, we obtain
$$[e_{z_1+m},e_{n}]=\sum_{h=2}^{m+q+1}P_h([e_{z_1+m-1},e_{n}]+[e_{z_1+m},e_{n-1}])e_{h+1}+\beta_{m,n-z_2}e_2$$
$$[[e_{z_1+m},e_n],e_{n-1}]=\sum_{h=2}^{m+q+1}P_h([e_{z_1+m-1},e_{n}]+[e_{z_1+m},e_{n-1}])[e_{h+1},e_{n-1}]+\beta_{m,n-z_2}[e_2,e_{n-1}]$$
$$=\sum_{h=z_1-1}^{m+q+1}P_h([e_{z_1+m-1},e_{n}]+[e_{z_1+m},e_{n-1}])[e_{h+1},e_{n-1}].$$

Consequently, Jacobi identity $J(e_{z_1+m},e_{n-1},e_n)=0$ is equivalent to
$$\sum_{h=z_1-1}^{m+q}P_h([e_{z_1+m-1},e_{n-1}]+[e_{z_1+m},e_{n-2}])[e_{h+1},e_n]-\sum_{h=z_1}^{q+r+1} P_h([e_{n-2},e_n])[e_{z_1+m},e_{h+1}]$$
$$-\sum_{h=z_1-1}^{m+q+1}P_h([e_{z_1+m-1},e_{n}]+[e_{z_1+m},e_{n-1}])[e_{h+1},e_{n-1}]=0.$$
According to equality (\ref{comb2}), we obtain that
$$P_{m+q+1}([e_{z_1+m},e_{n-1}]=\binom{m+q-1}{m} \gamma_1+\binom{m+q-1}{m-1} \alpha_2$$
$$P_{m+p+2}([e_{m+q+1},e_n])=\binom{m+p}{q}\gamma_1+\binom{m+p}{q+1} \alpha_2$$
$$P_{m+q+2}([e_{z_1+m},e_{n}]=\binom{m+q}{m}\gamma_1+\binom{m+q}{m-1}\alpha_2$$
$$P_{m+p+2}([e_{m+q+2},e_{n-1}])=\binom{m+p}{q-1}\gamma_1+\binom{m+p}{q}\alpha_2.$$
Next, by using equality (\ref{comb2}) and combinatorics, it is easy to prove that
$$P_{q+r+1}([e_{n-2},e_n])=
\gamma_1 \displaystyle\sum_{k=0}^{\lfloor\frac{r}{2}\rfloor} R_k+\alpha_2\displaystyle\sum_{k=0}^{\lfloor\frac{r}{2}\rfloor} S_k$$ and
$$P_{m+p+2}([e_{z_1+m},e_{q+r+2}])=\binom{m+p}{m}\gamma_1+\binom{p+m}{m-1}\alpha_2.$$
Notice that, by the choice of $s$, all the subindexes in the previous expressions of $P_h$ are bounded above by $n$.
Consequently, the coefficient of $e_{m+p+2}$ in the first term of the Jacobi identity is given by
$$\biggr\{\binom{m+q-1}{m} \gamma_1+\binom{m+q-1}{m-1} \alpha_2\biggr\}
\biggr\{\binom{m+p}{q}\gamma_1+\binom{m+p}{q+1} \alpha_2\biggr\},$$
while the coefficient of $e_{m+p+2}$ in the second term is
$$-\biggr\{\gamma_1 \displaystyle\sum_{k=0}^{\lfloor\frac{r}{2}\rfloor} R_k+\alpha_2\displaystyle\sum_{k=0}^{\lfloor\frac{r}{2}\rfloor} S_k\biggr\} \biggr\{ \binom{m+p}{m}\gamma_1+\binom{m+p}{m-1}\alpha_2  \biggr\}$$
and the coefficient of $e_{m+p+2}$ in the third term is
$$-\biggr\{\binom{m+q}{m} \gamma_1+\binom{m+q}{m-1} \alpha_2\biggr\}
\biggr\{\binom{m+p}{q-1}\gamma_1+\binom{m+p}{q} \alpha_2\biggr\}.$$
Therefore, the coefficient of $e_{m+p+2}$ in the Jacobi identity $J(e_{z_1+m},e_{n-1},e_n)=0$ is as follows
$$\biggr\{\binom{m+q-1}{m} \gamma_1+\binom{m+q-1}{m-1} \alpha_2\biggr\}
\biggr\{\binom{m+p}{q}\gamma_1+\binom{m+p}{q+1} \alpha_2\biggr\}$$
$$-\biggr\{\gamma_1 \displaystyle\sum_{k=0}^{\lfloor\frac{r}{2}\rfloor} R_k+\alpha_2\displaystyle\sum_{k=0}^{\lfloor\frac{r}{2}\rfloor} S_k\biggr\} \biggr\{ \binom{m+p}{m}\gamma_1+\binom{m+p}{m-1}\alpha_2  \biggr\}$$
$$-\biggr\{\binom{m+q}{m} \gamma_1+\binom{m+q}{m-1} \alpha_2\biggr\}
\biggr\{\binom{m+p}{q-1}\gamma_1+\binom{m+p}{q} \alpha_2\biggr\}=$$

$$a_m \gamma_1^2 + b_m \gamma_1 \alpha_2 + c_m \alpha_2^2.$$

\end{proof}

\begin{notation}\label{Fac}
Recall that we have denoted by $\mathcal{F}$ the family of laws of (non-model) filiform Lie algebras with associated triple $(z_1,z_2,n)$, see Notation \ref{Fabc}.  We denote by $\mathcal{G}$ the intersection of $\mathcal{F}$ with the linear subspace defined by the set of equations $\beta_{k \ell}=\gamma_{k+\ell-1}$ \, if \, $k+\ell \leq n-z_2$. The family $\mathcal{G}$ is a non-empty open set in an affine space of dimension $n-z_1$ corresponding to the parameters $\{\alpha_i\}$ and $\{\gamma_j\}$.

In \cite[Table 1]{CNT}, there is a description of the family $\mathcal{G}$ up to dimension $14$.
\end{notation}

\begin{rmr}\label{wewillsee}
We will see in the proof of Theorem \ref{solvability-index-leq4} that if $(z_1,z_2,n)$ satisfies
$$4\leq z_1 \leq 2(n-z_2)-4 \qquad {\rm and} \qquad z_1 \leq z_2 \leq n-3 \leq 2z_2-5,$$
then ${\mathcal G}$ is the empty set, see Remark \ref{fabcempty}.
\end{rmr}

\begin{rmr}\label{dim}
As far as we know, there is no closed formula for the dimension of $\mathcal{G}$. Nevertheless, from the definition of $\mathcal{G}$ and using the affine dimension theorem inequality, we have the following inequality
$${\rm dim}(\mathcal{F}) \leq  {\rm dim}(\mathcal{G}) + \dfrac{(n-z_2-1)(n+z_2-2z_1)}{2}.$$
\end{rmr}

\begin{notation}\label{syk}
Recall that we have denoted $s=\min\{q+1,r-q-1\}$. In the next Lemma we are going to use $\kappa=\lfloor \frac{s+p}{2} \rfloor$.
\end{notation}

\begin{lmm}\label{lemageneralizaado}

Let $\fkg$ be a $n$-dimensional non-model filiform Lie algebra belonging to the family $\mathcal{G}$ and associated with the triple $(z_1,z_2,n)$, whose derived algebra $\mathcal{D}\,\fkg$ is not abelian and satisfying $\gamma_i=\alpha_{i+1}=0$, for $i=1, \ldots, \kappa$. Then, for $m=0, \ldots, s$, the coefficient of the vector  $e_{m+p+4-2\kappa}$ in the Jacobi identity $J(e_{z_1+m},e_{n-1},e_n)=0$  is
$$a_m \gamma_{\kappa}^2+b_m \gamma_{\kappa} \alpha_{{\kappa}+1} + c_m \alpha_{{\kappa}+1}^2$$
where $a_m, b_m$ and $c_m$ are defined in Notation  \ref{ambmcm}.

\end{lmm}

\begin{proof}
Notice that we can assume $\kappa \geq 2$ since the cases $\kappa=0,1$ were dealt with in previous Lemmas. Moreover, the proof is similar to the one of Lemma \ref{caso(z1+m,n-1,n)}.

\end{proof}

\begin{thr}\label{solvability-index-leq4}

Any $n$-dimensional filiform Lie algebra in the family $\mathcal{G}$ has derived length at most $3$.

\end{thr}

\begin{proof}

Let $\fkg$ be a Lie algebra belonging to the family $\mathcal{G}$ with associated triple $(z_1,z_2,n)$. If $z_2=n-1$, then $\fkg$ is metabelian according to Remark \ref{cason-1}. In case $z_2=n-2$, the algebra $\fkg$ has derived length $3$ by Lemma \ref{cason-2}.

We assume now that $\fkg$ is a non-model filiform Lie algebra with $z_2 \leq n-3$. In that case, $n \geq 8$ due to inequalities (\ref{Bound}) and the derived Lie algebra $\mathcal{D} \fkg$ is not abelian.

Lemma \ref{alpha1es0} implies $\alpha_1=0$ in the law given in Theorem \ref{leygeneral}.
Therefore, in $\mathcal{D}^2\fkg$ we have the following three types of brackets

\vspace{.4cm}

\noindent $[e_{z_1+i},e_{z_2+1}]=\alpha_2 e_{i+1}+\cdots + \alpha_{i+1}e_2 \  {\mbox{ for  }} \, 0 \leq i \leq z_2-z_1$;

\vspace{.4cm}

\noindent $[e_{z_1},e_{z_2+j}]=\gamma_1 \, e_j+\cdots+\gamma_{j-1} \, e_2 \ {\mbox{ for  }} \, 2 \leq j \leq n-z_2-1$;

\vspace{.4cm}

\noindent $[e_{z_1+k},e_{z_2+\ell}] = \sum_{h=2}^{k+\ell} P_h([e_{z_1+k-1},e_{z_2+\ell}] +[e_{z_1+k},e_{z_2+\ell-1}])
e_{h+1}+ \gamma_{k+\ell-1} \, e_2$ \, {\mbox{ for }}\, $2 \leq \ell \leq n-z_2-1$, and  $0< k < z_2-z_1+\ell$.

\vspace{.4cm}

Consequently, $D^2\fkg$ is contained in the greatest of the following spaces
$$\langle e_2, \ldots, e_{z_2-z_1+2} \rangle, \quad \langle e_2, \ldots,e_{n-z_2} \rangle, \quad \langle e_2, \ldots, e_{2n-z_1-z_2-2} \rangle.$$

Notice that $z_2-z_1+2 \leq 2n-z_1-z_2-2$ since $z_2 \leq n-3$. Moreover, $n-z_2 \leq 2n-z_1-z_2-2$ due to the fact that $n-z_1 \leq n-z_2 \leq 2$.

Therefore, $D^2\fkg\subset \langle e_2, \ldots, e_{2n-z_1-z_2-2} \rangle$.
Let us notice that if $2n-z_1-z_2-2 < z_2+1$, then $D^3\fkg=\{0\}$ by definition of $z_2$, that is, when $z_1 > 2(n-z_2)-4$ we can affirm that $\mathfrak{g}$ has derived length $3$.

Now, we suppose that $4 \leq z_1 \leq 2(n-z_2)-4$ where $z_2 \leq n-3$ and verifying expression (\ref{Bound}). We are going to prove that there is no filiform Lie algebra associated with the triple $(z_1,z_2,n)$ under these conditions.

In order to do so, we consider the Jacobi identities $J(e_{z_1+m},e_{n-1},e_n)=0$, for $0 \leq m \leq s$, where $s=\min\{q+1,r-q-1\}$ (see Notation \ref{syk}). We prove, inductively, that, at least, one of the following brackets is zero
$$[e_{z_1},e_n]=\gamma_1 \, e_{n-z_2}+\gamma_2 \, e_{n-z_2-1}+\cdots+\gamma_{n-z_2-1} \, e_2,$$
$$[e_{z_2},e_{z_2+1}]=\alpha_2 e_{z_2-z_1+1}+\cdots + \alpha_{z_2-z_1+1}e_2$$
which is a contradiction with the definition of the invariants $z_1$ and $z_2$.

First, we consider the equations defined by the coefficient of $e_{m+p+2}$ in the Jacobi identities $J(e_{z_1+m},e_{n-1},e_n)=0$, for $m=0,1$. Notice that those expressions were obtained in Lemma \ref{caso(z1+m,n-1,n)} and are given by
\begin{equation}\label{mvale0}
\gamma_1^2 a_0+\gamma_1 \alpha_2 b_0=0
\end{equation}
\begin{equation}\label{mvale1}
\gamma_1^2 a_1+\gamma_1 \alpha_2 b_1 + \alpha_2^2 c_1=0,
\end{equation}
where
$$a_0=\binom{p}{q}-\binom{p}{q-1} - \sum_{k=0}^{\lfloor \frac{r}{2} \rfloor} R_k=\dfrac{(p-2q+1)}{(p-q+1)}\binom{p}{q}-\sum_{k=0}^{\lfloor \frac{r}{2} \rfloor} R_k$$
$$b_0=\binom{p}{q+1}-\binom{p}{q} - \sum_{k=0}^{\lfloor \frac{r}{2} \rfloor} S_k=\dfrac{(p-2q-1)}{(q+1)}\binom{p}{q}-\sum_{k=0}^{\lfloor \frac{r}{2} \rfloor} S_k$$
$$a_1=q\binom{p+1}{q}-(q+1)\binom{p+1}{q-1} -(p+1) \sum_{k=0}^{\lfloor \frac{r}{2} \rfloor} R_k=\dfrac{(p-2q+1)}{(p-q+2)}q\binom{p+1}{q}-(p+1)\sum_{k=0}^{\lfloor \frac{r}{2} \rfloor} R_k$$
$$b_1=q\binom{p+1}{q+1}-\binom{p+1}{q}-(q+1)\binom{p+1}{q}-\binom{p+1}{q-1} - \sum_{k=0}^{\lfloor \frac{r}{2} \rfloor} R_k-(p+1)\sum_{k=0}^{\lfloor \frac{r}{2} \rfloor} S_k =$$
$$\dfrac{(q(p-2q-2)-2)(p-q+2)-q(q+1)}{(q+1)(p-q+2)}\binom{p+1}{q}-\sum_{k=0}^{\lfloor \frac{r}{2} \rfloor} R_k-(p+1)\sum_{k=0}^{\lfloor \frac{r}{2} \rfloor} S_k$$
$$c_1=\binom{p+1}{q+1}-\binom{p+1}{q} - \sum_{k=0}^{\lfloor \frac{r}{2} \rfloor} S_k=\dfrac{p-2q}{p-q+1}\binom{p+1}{q+1}-\sum_{k=0}^{\lfloor \frac{r}{2} \rfloor} S_k.$$
Now, we compute the difference between the equation (\ref{mvale0}) multiplied by $(p+1)$ and equation (\ref{mvale1}), obtaining
\begin{equation}\label{m0ym1restados}
\gamma_1^2 a' + \gamma_1 \alpha_2 b' - \alpha_2^2 c_1=0
\end{equation}
where
$$a'=\dfrac{(p-2q+1)(p-2q+2)}{p-q+2} \binom{p+1}{q}$$ and
$$b'=\dfrac{(p-q+2)\{(p-2q)^2+(q+1)\} +q(q+1)}{(p-q+2)(q+1)} \binom{p+1}{q}+\sum_{k=0}^{\lfloor \frac{r}{2} \rfloor} R_k.$$
Let us consider the system formed by equations (\ref{mvale0}) and (\ref{m0ym1restados}).
Notice that if $\gamma_1=0$, then we conclude that $\alpha_2=0$ since $c_1\neq 0$. Analogously, if $\gamma_2=0$, then we obtain that $\gamma_1=0$ due to the fact that $a_0 \neq 0$. Consequently, we suppose that $\gamma_1 \neq 0$ and $\alpha_2 \neq 0$. From equation (\ref{mvale0}), we obtain
$$\gamma_1=-\frac{b_0}{a_0}\alpha_2.$$
Now, we substitute the previous equality in equation (\ref{m0ym1restados}), obtaining
$$\alpha_2^2 \left(\frac{b_0^2 a'}{a_0^2}-\frac{b_0 b'}{a_0}-c_1   \right)=0.$$
Therefore, we obtain the expression
\begin{equation}\label{expresionm0m1}
b_0^2 a' = a_0 (a_0 c_1-b_0 b').
 \end{equation}
Next, we analyze the sign of the constants $a_0,b_0,a',b',c_1$. By using the definition of $p$, $q$, $r$ (see Notation \ref{pqrRkSk}) and expression (\ref{Bound}), it is easy to prove that
$$a_0<0, \quad b_0<0, \quad a'\geq0, \quad b'>0, \quad c_1<0.$$
Let us notice that in both cases, we obtain a contradiction with expression (\ref{expresionm0m1}). We conclude that
$$\gamma_1=\alpha_2=0.$$

Next, for the inductive step, we assume that $\gamma_k=\alpha_{k+1}=0$, for $1 < k < s$, where $s=\min\{q+1,r-q-1\}$. Now, we impose the coefficient of $e_{m+p+4-2s}$ in the Jacobi identity $J(e_{z_1+m},e_{n-1},e_n)=0$, for $m=s-1$ and $m=s$, to be zero. Let us notice that those coefficients were obtained in Lemma \ref{lemageneralizaado}. Following a similar reasoning than before, we conclude that
$$\gamma_{s}=\alpha_{s+1}=0.$$
Consequently, we obtain that $[e_{z_1},e_n]=0$ and/or $[e_{z_2},e_{z_2+1}]=0$ and this is a contradiction with the definition of invariants $z_1$ and $z_2$, see Subsection \ref{TwoNumerical}.

\end{proof}

\begin{rmr}\label{counterexample} The converse of Theorem \ref{solvability-index-leq4} does not hold. The Lie algebra $\frak{f}_{\frac{9}{10},7}$ in \cite[Prop. 3.1]{BUarxiv}, has derived length 3 and does not belong to the family $\mathcal{G}$.
\end{rmr}

\begin{rmr}\label{fabcempty}
Notice that we have proved in particular that if $(z_1,z_2,n)$ satisfies
$$4\leq z_1 \leq 2(n-z_2)-4 \qquad {\rm and} \qquad z_1 \leq z_2 \leq n-3 \leq 2z_2-5,$$
then $\mathcal{G}$ is the empty set. The number of triples $(z_1,z_2,n)$ in the previous region is asymptotically $\frac{n^2}{3}$.

\end{rmr}

\begin{rmr}

In \cite[Example 3.2]{BUarxiv} an example of a filiform Lie algebra with dimension $15$ and derived length $4$ is given. Using Theorem \ref{leygeneral} we can generalize this example providing a family of $15$-dimensional filiform Lie algebras with derived length $4$ associated with the triple $(4,9,15)$. This family is given by the following law:
\begin{align*}\scriptsize
[e_1,e_h] &=e_{h-1} \,\, {\rm for} \,\, 3 \leq h \leq 15, & [e_4,e_{15}] &=\gamma_5 e_2, \\
[e_{6}, e_{15}] &= (\gamma_5+2\beta_{1,5}+\beta_{2,4})e_{4}, & [e_{6}, e_{13}] &= \beta_{2,4}e_{2}, \\
[e_{7}, e_{13}] &= (\beta_{2,4}+\beta_{3,3})e_{3}, & [e_5, e_{14}] &= \beta_{1,5}e_{2},   \\
[e_{7}, e_{15}] &= (\gamma_5+3\beta_{1,5}+3\beta_{2,4}+\beta_{3,3})e_{5}, &  [e_{5}, e_{15}] &= (\gamma_5+\beta_{1,5})e_{3}, \\
[e_{8}, e_{13}] &= (\beta_{2,4}+2\beta_{3,3}+\beta_{4,2})e_{4}, & [e_{6}, e_{14}] &= (\beta_{1,5}+\beta_{2,4})e_{3}, \\
[e_{8}, e_{14}] &= (\beta_{1,5}+3\beta_{2,4}+3\beta_{3,3}+\beta_{4,2})e_{5}, &  [e_{7}, e_{12}] &= \beta_{3,3}e_{2}, \\
[e_{8}, e_{15}] &= (\gamma_5+4\beta_{1,5}+6\beta_{2,4}+4\beta_{3,3}+\beta_{4,2})e_{6}   , &  [e_{7}, e_{14}] &= (\beta_{1,5}+2\beta_{2,4}+\beta_{3,3})e_{4}, \\
[e_{9}, e_{12}] &= (\beta_{3,3}+2\beta_{4,2}+\alpha_6)e_{4}, & [e_{8}, e_{11}] &= \beta_{4,2}e_{2}, \\
[e_{9}, e_{13}] &= (\beta_{2,4}+3\beta_{3,3}+3\beta_{4,2}+\alpha_6)e_{5}, & [e_{8}, e_{12}] &= (\beta_{3,3}+\beta_{4,2})e_{3}, \\
[e_{9}, e_{14}] &= (\beta_{1,5}+4\beta_{2,4}+6\beta_{3,3}+4\beta_{4,2}+\alpha_6)e_{6}, & [e_{9}, e_{10}] &= \alpha_6e_{2}, \\
[e_{9}, e_{15}] &= (\gamma_5+5\beta_{1,5}+10\beta_{2,4}+10\beta_{3,3}+5\beta_{4,2}+\alpha_6)e_{7}, & [e_{9}, e_{11}] &= (\beta_{4,2}+\alpha_6)e_{3}, \\
[e_{10}, e_{12}] &= (\beta_{3,3}+3\beta_{4,2}+2\alpha_6)e_{5}, & [e_{10}, e_{11}] &= (\beta_{4,2}+\alpha_6)e_{4},\\
[e_{10}, e_{13}] &= (\beta_{2,4}+4\beta_{3,3}+6\beta_{4,2}+3\alpha_6)e_{6}, & \\
[e_{10}, e_{14}] &= (\beta_{1,5}+5\beta_{2,4}+10\beta_{3,3}+10\beta_{4,2}+4\alpha_6)e_{7}, & \\
[e_{10}, e_{15}] &= (\gamma_5+6\beta_{1,5}+15\beta_{2,4}+20\beta_{3,3}+15\beta_{4,2}+5\alpha_6)e_{8}, & \\
[e_{11}, e_{12}] &= (\beta_{3,3}+3\beta_{4,2}+2\alpha_6)e_{6}, & \\
[e_{11}, e_{13}] &= (\beta_{2,4}+5\beta_{3,3}+9\beta_{4,2}+5\alpha_6)e_{7}, & \\
[e_{11}, e_{14}] &= (\beta_{1,5}+6\beta_{2,4}+15\beta_{3,3}+19\beta_{4,2}+9\alpha_6)e_{8}, & \\
[e_{11}, e_{15}] &= (\gamma_5+7\beta_{1,5}+21\beta_{2,4}+35\beta_{3,3}+34\beta_{4,2}+14\alpha_6)e_{9}, & \\
[e_{12}, e_{13}] &= (\beta_{2,4}+5\beta_{3,3}+9\beta_{4,2}+5\alpha_6)e_{8}, & \\
[e_{12}, e_{14}] &= (\beta_{1,5}+7\beta_{2,4}+20\beta_{3,3}+28\beta_{4,2}+14\alpha_6)e_{9}, & \\
[e_{12}, e_{15}] &= (\gamma_5+8\beta_{1,5}+28\beta_{2,4}+55\beta_{3,3}+62\beta_{4,2}+28\alpha_6)e_{10}, &  \\
[e_{13}, e_{14}] &= (\beta_{1,5}+7\beta_{2,4}+20\beta_{3,3}+28\beta_{4,2}+14\alpha_6)e_{10}, & \\
[e_{13}, e_{15}] &= (\gamma_5+9\beta_{1,5}+35\beta_{2,4}+75\beta_{3,3}+90\beta_{4,2}+42\alpha_6)e_{11}, & \\
[e_{14}, e_{15}] &= (\gamma_5+9\beta_{1,5}+35\beta_{2,4}+75\beta_{3,3}+90\beta_{4,2}+42\alpha_6)e_{12}. &
\end{align*}
verifying the following relations
$$\alpha_6 = \frac{1}{14} \beta, \quad \gamma_5 = 363 \beta, \quad \beta_{1,5}= 27 \beta, \quad \beta_{2,4}=\frac{21}{5} \beta, \quad \beta_{4,2}=\frac{3}{10} \beta$$
where $\beta=\beta_{3,3} \in {\C}\setminus\{0\}$.
Notice that \cite[Example 3.2]{BUarxiv} is obtained for $\beta=\frac{1}{858}$.

Finally, it is also possible, using Theorem \ref{leygeneral}, to give a general family of filiform Lie algebras with dimension $31$ and derived length $5$. The laws of the corresponding family $\mathcal{F}$ associated with the triple $(4,17,31)$ depend on the following $14$ parameters
$$\{\alpha_{14}, \quad \gamma_{13}, \quad  \beta_{k,\ell} \,:  \,\, k+\ell=14, \,\, 1\leq k, \,\, 2\leq \ell  \}$$
and the derived series of this algebra $\fkg$ is given by
$$\mathcal{D}^0 \, \fkg=\fkg, \quad \mathcal{D} \, \fkg=\langle e_2, \ldots, e_{30} \rangle, \quad \mathcal{D}^2 \, \fkg=\langle e_2, \ldots, e_{26} \rangle,
\quad \mathcal{D}^3 \, \fkg=\langle e_2, \ldots, e_{18} \rangle,$$
$$\mathcal{D}^4 \, \fkg=\langle e_2 \rangle, \quad \mathcal{D}^5 \, \fkg=\{0\}.$$
\end{rmr}

\section*{Acknowledgment}
The authors would like to thank the referees for their valuable comments that have improved and clarified the presentation of the paper. This work has been partially supported by MTM2016-75024-P and FEDER, FQM-326 and FQM-333.

\end{document}